\documentclass[reqno,11pt]{amsart}
\usepackage{a4wide,color,eucal,enumerate,mathrsfs}
\usepackage[normalem]{ulem}
\usepackage{psfrag}          
\usepackage{amsmath,amssymb,epsfig}
\numberwithin{equation}{section}

\usepackage[latin1]{inputenc}

\usepackage[pdfborder={0 0 0}]{hyperref}

\newtheorem{theorem}{Theorem}[section]

\newtheorem{lemma}[theorem]{Lemma}
\newtheorem{proposition}[theorem]{Proposition}

\theoremstyle{definition}
\newtheorem{definition}[theorem]{Definition}

\newtheorem{example}[theorem]{Example}

\newcommand{\diam}{\mathop{\rm diam}\nolimits}
\newcommand{\dist}{\mathop{\rm dist}\nolimits}
\newcommand{\supp}{\mathop{\rm supp}\nolimits}
\def\esup{\mathop\mathrm{ess\,sup\,}}
\renewcommand{\d}{{\mathrm d}}
\newcommand{\N}{\mathbb{N}}

\newcommand{\R}{\mathbb{R}}

\newcommand{\sfd}{{\sf d}}
\newcommand{\prob}[1]{\mathscr P(#1)}

\newcommand{\eps}{\varepsilon}

\newcommand{\id}{{\rm{id}}}

\newcommand{\restr}[1]{\lower3pt\hbox{$|_{#1}$}}

\begin{document}

\title{
$L^\infty$ estimates in optimal mass transportation}

\author{Heikki Jylh\"a}
\author{Tapio Rajala}
\address{University of Jyvaskyla, Department of Mathematics and Statistics,
P.O.Box 35 (MaD), FI-40014 University of Jyvaskyla, Finland
}
\email{heikki.j.jylha@jyu.fi}
\email{tapio.m.rajala@jyu.fi}
\thanks{This research was started while the first author was taking part in the Junior Hausdorff Trimester Program Optimal Transportation at the Hausdorff Research Institute for Mathematics in Bonn. The first author would like to thank the Institute for gracious support throughout the program. T.R. acknowledges the support of the Academy of Finland project no. 137528.}

\begin{abstract}
 We show that in any complete metric space the probability measures $\mu$ with compact and connected support
 are the ones having the property 
 that the optimal tranportation distance to any other probability measure $\nu$ living on the support of $\mu$
 is bounded below by a positive function of the $L^\infty$ transportation distance between $\mu$ and $\nu$.
 The function giving the lower bound depends only on the lower bound of the $\mu$-measures of balls centered at the
 support of $\mu$ and on the cost function used in the optimal transport. We obtain an essentially sharp
 form of this function.
 
 In the case of strictly convex cost functions we show that a similar estimate holds on the level of
 optimal transport plans if and only if the support of $\mu$ is compact and sufficiently close to being geodesic.
 
 We also study when convergence of compactly supported measures in $L^p$ transportation distance implies
 convergence in $L^\infty$ transportation distance. For measures with connected supports this property is
 characterized by uniform lower bounds on the measures of balls centered at the supports of the measures
 or, equivalently, by the Hausdorff-convergence of the supports.
\end{abstract}

\maketitle



\section{Introduction}

Suppose we are given two Borel probability measures $\mu,\nu \in \prob{X}$ on a metric space $(X,\sfd)$
and a function $c \colon X \times X \to [-\infty,\infty]$ representing the cost of moving mass.
The optimal mass transportation problem, in the Kantorovich formulation, is then to minimize the quantity
\begin{equation}\label{eq:Kantorovich}
 \int_{X\times X} c(x,y)\,\d\lambda(x,y)
\end{equation}
over all possible transport plans
$\lambda \in \Pi(\mu,\nu)$, i.e. Borel probability measures in $X \times X$ having the marginals $\mu$ and $\nu$.
An optimal transport plan $\lambda$ minimizing \eqref{eq:Kantorovich} exists under mild regularity assumptions,
for example if the cost function $c$ is lower semicontinuous and bounded from below and if the metric space $(X,\sfd)$
is complete and separable \cite[Theorem 4.1]{Villani09}.
Under much more restrictive assumptions such minimizer is unique and given by an optimal transport map $T \colon X \to X$
as $\lambda = (\id, T)_\sharp\mu$.

Often the cost function in \eqref{eq:Kantorovich} is of the form $c(x,y)=h(\sfd(x,y))$
with some convex function $h \colon [0,\infty) \to [0,\infty)$.
The most commonly used cost functions are the $p$:th powers of the distance
with $p \in [1,\infty)$. This leads to the $L^p$ transportation distances $W_p$ defined
between $\mu,\nu \in \prob{X}$ by
\[
 W_p(\mu,\nu) = \inf_{\lambda \in \Pi(\mu,\nu)}\left(\int \sfd^p(x,y)\,\d\lambda(x,y)\right)^{1/p}.
\]
It is well known that the $W_p$ distance metrizes the topology of weak convergence (up to convergence of $p$:th moments).
The $W_p$ distances with $p \in (1,\infty)$ are often easier to handle, for instance due to strict convexity, than the limiting
cases $p = 1$ and $p = \infty$. In the latter one the distance is defined as
\[
 W_\infty(\mu,\nu) = \inf_{\lambda \in \Pi(\mu,\nu)} \lambda-\esup_{(x,y) \in X^2} \sfd(x,y). 
\]
The distance $W_\infty$ is even more cumbersome than $W_1$. This is because the problem of infimizing the cost 
\[
 \lambda-\esup_{(x,y) \in X^2} \sfd(x,y)
\]
over all $\lambda \in \Pi(\mu,\nu)$ is not
convex and thus it is not additive. Consequently, restrictions of optimal transports for $W_\infty$ are not necessarily
optimal. The problem of restrictions in $W_\infty$ was addressed by Champion, De Pascale and Juutinen in \cite{CDPJ2008} where they
introduced the notion of \emph{restrictable solutions}. Those are the optimal transports that retain optimality under restrictions.
Restrictable solutions appear as the limit solutions in the approximation as $p\to\infty$ and, more generally, can be characterized by a suitable version of cyclical monotonicity. The results of \cite{CDPJ2008}
were later generalized by the first author of this paper in \cite{J2015}. 

Despite the problematic features of the $W_\infty$ distance it is still used in many areas of mathematics.
The topology induced by $W_\infty$ is a natural one to work with to study local minimizers of certain functionals, such as the energy associated with the astrophysical fluid model considered by McCann \cite{M2006} and the interaction energy considered by Balagu\'e, Carrillo, Laurent and Raoul in \cite{BCLR2013a} and \cite{BCLR2013b}.
Recently, the $W_\infty$ distance has been used in quantum physics by Busch, Lahti and Werner in \cite{BLW2014},
and in the study of convergence of empirical measures by Garcia Trillos and Slepcev in \cite{GTS2014}.
It has also been used in heat flow estimates, see for instance the papers by
Kuwada \cite{K2010} and Savar\'e \cite{S2014}, as well as in BV-theory by Ambrosio and Di Marino \cite{ADM2014}.
\medskip

It is easy to see that we always have the inequality
\begin{equation}\label{eq:trivial}
 W_p(\mu,\nu) \le W_\infty(\mu,\nu)
\end{equation}
for all $p \ge1$. In general no inequality converse to \eqref{eq:trivial} holds. 
In fact, one almost immediately notices that unlike $W_p$ for finite $p$, the distance 
$W_\infty$ no longer gives the weak topology if the space has more than one point.
In this paper we consider the questions when exactly does the convergence in $W_p$ imply convergence in $W_\infty$ 
and, in particular, when is it possible to get a reverse inequality to \eqref{eq:trivial}
in some uniform and quantitative form.

One answer to the latter question was given by Bouchitt\'e, Jimenez, and Rajesh in \cite{BJR2007}. 
They proved that for a bounded Lipschitz domain $\Omega \subset \R^d$,
for any $p>1$, and for any $\mu = f\mathcal{L}^d\restr{\Omega} \in \prob{\R^d}$ one has
\begin{equation}\label{eq:BJR}
 W_p^p(\mu,\nu)\ge \frac{C(p,d,\Omega)}{\|f^{-1}\|_{L^\infty}}W_\infty(\mu,\nu)^{p+d} \qquad\text{for all }\nu \in \prob{\overline{\Omega}}.
\end{equation}
Their result left open the question what happens in the limit case $p=1$ when $d>1$.
We show that \eqref{eq:BJR} also holds for $p=1$, as was conjectured in \cite{BJR2007}. This will be an immediate corollary of the following
general result that characterizes in metric spaces 
the measures $\mu$ for which there exists an estimate of the type \eqref{eq:BJR}.

{\begin{theorem}\label{thm:main}
 Let $(X,\sfd)$ be a complete metric space,
 $h \colon [0,\infty) \to [0,\infty)$ a nondecrasing function with $h(t)>0$ for all $t>0$
 and $\mu \in \prob{X}$.
 Then there exists a nondecreasing function $\omega \colon [0,\infty) \to [0,\infty)$ with $\omega(t)>0$ for all $t > 0$ such that
 \begin{equation}\label{eq:maineq}
  \inf_{\lambda \in \Pi(\mu,\nu)}\int h\circ \sfd\,\d\lambda \ge \omega(W_\infty(\mu,\nu)) \qquad \text{for all }\nu \in \prob{\supp\mu}
 \end{equation}
 if and only if $\supp\mu$ is compact and connected.

 Moreover, in such case one can take as $\omega$ in \eqref{eq:maineq} the function
 $\omega(t) = \tfrac12m({t}/{17})h({t}/{17})$, where 
 \[
  m(t) := \inf_{x \in \supp\mu}\mu(B(x,t)).
 \] 
\end{theorem}}

The function $\omega$ in Theorem \ref{thm:main} is essentially sharp in the sense that
it cannot be improved to a function larger than $\omega(t) = m(t)h(t)$, see Proposition \ref{prop:sharpness}.
In order to see that Theorem \ref{thm:main} implies \eqref{eq:BJR} notice that 
for a bounded Lipschitz domain $\Omega \subset \R^d$ and $\mu = f\mathcal{L}^d\restr{\Omega} \in \prob{\R^d}$
we have
\[
 m(t) \ge \frac{C(\Omega)}{\|f^{-1}\|_{L^\infty}}t^d \qquad\text{for all } 0 < t < \diam(\Omega).
\]
We also note that the condition $\nu\in\prob{\supp\mu}$ in \eqref{eq:maineq} is important: Take any $x\notin\supp\mu$ and $\nu_t=(1-t)\mu+t\delta_x$ for $t\in(0,1)$. Then it is easy to see that the right-hand side of \eqref{eq:maineq} is bounded from below by a positive constant, but the left side goes to zero as $t\to 0$. Thus \eqref{eq:maineq} cannot hold for all $\nu\in\prob{\supp\mu\cup\{x\}}$.


Our proof of Theorem \ref{thm:main} is quite different from the proof of \eqref{eq:BJR} in \cite{BJR2007}.
In \cite{BJR2007} it is proven that if $\Omega$ is convex, then \eqref{eq:BJR} holds with
$W_\infty$ replaced by the essential supremum of the transport distance in the optimal map.
This is then used to derive \eqref{eq:BJR}.
Instead in our proof of Theorem \ref{thm:main} the transport for estimating the $W_\infty$ distance will be a modification
of the transport appearing on the left-hand side of \eqref{eq:maineq}. 
The modification we use is intuitively quite obvious: the part that is transported long way with the original transport
will be redefined to be a combination of shorter distance transports. The rigorous modification is done in 
Lemma \ref{lma:bigpartlongway}.
It is clear that in the general case of
Theorem \ref{thm:main} such modification is necessary. Indeed, as was noted also in \cite{BJR2007},
for example in the class of optimal $W_1$ transports on the real line one cannot have
uniform $L^\infty$ estimates. This is due to the fact that non-monotone transports may also be optimal.
\medskip

As was mentioned above, in \cite{BJR2007} it was proven that for convex $\Omega$ one can get an $L^\infty$ estimate 
of the type \eqref{eq:BJR} for the optimal transport map.
In the context of our paper, the question is then: for strictly convex cost functions $h$
under what assumptions on $\supp\mu$ do we have an $L^\infty$ estimate of the type \eqref{eq:BJR} for optimal tranport plans?
It turns out that the existence of such estimate is characterized
by what could be called \emph{the strict $h\circ\sfd$-cyclical convexity of $\supp\mu$}, see condition \eqref{eq:eisykli}.
The idea behind the following Theorem \ref{thm:planestimate} is that the more convex $h$ is, the further from geodesic $\supp\mu$ is allowed to be.
As an easy example of this phenomenon,
consider the cost $\sfd^p(x,y)$ with $p>1$, as the metric space a snowflaked distance $\sfd(x,y) = |x-y|^{1/s}$
on the real line for some $s>1$, and as the support of the measure $\supp \mu = [0,1]$.
Then for $p>s$ there exists an $L^\infty$ estimate for optimal transport plans and below the critical case $p\le s$ there does not.
Of course this example is quite articifial, since it is equivalent to $L^{p/s}$ optimal transportation on the Euclidean real line.
However, it still shows how the convexity of $h$ affects the characterizing requirement on $\supp\mu$ for the existence of uniform 
$L^\infty$ estimates of optimal tranport plans.


\begin{theorem}\label{thm:planestimate}
 Let $(X,\sfd)$ be a complete metric space,
 $h \colon [0,\infty) \to [0,\infty)$ a nondecreasing continuous function with $h(t)>0$ for all $t>0$
 and $\mu \in \prob{X}$. Then the following conditions are equivalent 
\begin{enumerate}
\item The set $\supp\mu$ is compact, and for every $x,y\in\supp\mu$, $x\neq y$, there exists $N\in\N$ and a sequence $(z_i)_{i=0}^{N+1}\subset\supp\mu$ such that $z_0=x$, $z_{N+1}=y$ and
\begin{equation}\label{eq:eisykli}
\sum_{i=0}^N h(\sfd(z_i,z_{i+1})) \;<\; h(\sfd(x,y)).
\end{equation}
\item There exists a nondecreasing function $\omega \colon [0,\infty) \to [0,\infty)$ with $\omega(t)>0$ for all $t > 0$ such that the following holds: If we take any $\nu \in \prob{\supp\mu}$ and a transport plan $\lambda\in\Pi(\mu,\nu)$ that minimizes the functional $C_h \colon \Pi(\mu,\nu)\to[0,\infty]$,
\[
C_h(\tilde{\lambda})= \int h\circ \sfd(x,y)\,\d\tilde{\lambda}(x,y), 
\]
then we have
 \begin{equation}\label{eq:planeq}
  \int h\circ \sfd(x,y)\,\d\lambda(x,y) \ge \omega(\lambda-\esup_{(x,y) \in X^2} \sfd(x,y)). 
 \end{equation}
\end{enumerate}
 Moreover, if (1) holds then one can take as $\omega$ in \eqref{eq:planeq} the function
 $\omega(t) = m({\rho(t)}/{4})h({\rho(t)}/{4})$, where $m$ is as in Theorem \ref{thm:main} and $\rho \colon [0,\infty) \to [0,\infty)$, $\rho(t)>0$ for all $t>0$, satisfies the following slightly stronger version of \eqref{eq:eisykli}: For every $x,y\in\supp\mu$, $x\neq y$, there exist $N\in\N$ and a sequence $(z_i)_{i=0}^{N+1}\subset\supp\mu$ such that $z_0=x$, $z_{N+1}=y$ and
\[
\sum_{i=0}^Nh\big(\sfd(z_i,z_{i+1})+\rho(\sfd(x,y))\big) \;<\; h(\sfd(x,y)).
\] 
\end{theorem}


The condition (1) of Theorem \ref{thm:planestimate} is clearly satisfied by any strictly convex cost function $h$ if
$\supp\mu$ is geodesic. Thus Theorem \ref{thm:planestimate} can be seen as a generalization of the corresponding result in \cite{BJR2007}.
The condition (1) is also satisfied if we have $h(0) = h'(0)=0$ and if any two points in $\supp \mu$ can be connected by a rectifiable curve
in $\supp \mu$. This is a special case of the more general result presented in Proposition \ref{prop:nicelyconnectedsupport}.
\medskip

The modification of transports needed in the proof of Theorem \ref{thm:main} that is done in Lemma \ref{lma:bigpartlongway}
also gives the following result on convergence in different topologies.
By $\sfd_H$ we denote the Hausdorff-distance between closed sets in $(X,\sfd)$
defined as
\[
 \sfd_H(A,B) := \max\big(\sup_{x \in A}\dist(x,B),\sup_{y \in B}\dist(y,A)\big).
\]

\begin{theorem}\label{thm:convergence}
 Suppose $(\mu_i)_{i=1}^\infty$ is a sequence of compactly supported 
 probability measures in a complete metric space $(X,\sfd)$ and let $p \ge 1$.
 Then $W_\infty(\mu_i,\mu) \to 0$ if and only if the following conditions hold
 \begin{enumerate}
  \item $W_p(\mu_i,\mu) \to 0$.
  \item $\sfd_{H}(\supp \mu_i,\supp \mu) \to 0$ as $i \to \infty$.
  \item If there exist sequences of positive measures $(\mu_i^1)_{i=1}^\infty$ and $(\mu_i^2)_{i=1}^\infty$ such that
        $\mu_i = \mu_i^1 + \mu_i^2$ for all $i$, $W_p(\mu_i^1,\mu^1) \to 0$ and $W_p(\mu_i^2,\mu^2) \to 0$
        for some measures $\mu^1$ and $\mu^2$, and if $\inf_{i\in\N}\dist(\supp\mu_i^1,\supp\mu_i^2)>0$,
        then there exists $i_0 \in \N$ such that $\mu_i^1(X) = \mu^1(X)$ for all $i \ge i_0$.
 \end{enumerate}
\end{theorem}

A few comments on the formulation of Theorem \ref{thm:convergence} are in order.
First of all, due to compactness the condition (2) in Theorem \ref{thm:convergence} can be replaced by the 
requirement of uniform lower bounds for measures of balls:
\begin{itemize}
 \item[(2')] $\inf_{i \in \N}\inf_{x \in \supp \mu_i}\mu_i(B(x,r)) > 0$ for all $r>0$.
\end{itemize}
See Lemma \ref{lma:condition2equiv} for the proof of this.
Secondly, the condition (3) in Theorem \ref{thm:convergence}
is needed to handle the case where $\supp \mu$ has many connected components. 
The idea of (3) is that the measure of a component has to stabilize to a constant in the convergence.
The condition (3) has to be stated in terms of separated parts of supports, since it could well be
that every connected component of the support has zero measure. Consider for example the case
with $\supp \mu$ a Cantor set and $\mu$ the corresponding Hausdorff measure restricted to this set.
If we assume $\supp \mu$ to be connected or each $\supp\mu_i$ to be connected,
the condition (3) can be dropped.

Theorem \ref{thm:convergence} does not hold if we drop the compactness assumption on $\supp \mu_i$.
For example, we can have $\supp \mu_i = \R$ for all $i$, $\supp\mu = \R$, $W_p(\mu_i,\mu) \to 0$
as $i \to \infty$ for all $1 < p < \infty$, but $W_\infty(\mu_i,\mu) = \infty$ for all $i$.
See Example \ref{ex:noncompact}.

The rest of this paper is organized as follows. In Section 2 we study lower bounds for transport cost in terms of the $L^\infty$ transportation distance. The goal of the section is to prove Theorems \ref{thm:main} and \ref{thm:convergence}. In Section 3 we consider lower bounds for transport cost in terms of the maximal transportation distance in the optimal transport plan. As our main result in this section we prove Theorem \ref{thm:planestimate}.

\section{Comparison of infima}

In this section we prove Theorems \ref{thm:main} and \ref{thm:convergence}. A key ingredient in both proofs
is Lemma \ref{lma:bigpartlongway}. Before stating and proving it we start with an easier lemma.

\begin{lemma}\label{lma:mexists}
 Let $(X,\sfd)$ be a complete metric space and $\mu \in \prob{X}$. Then $\supp\mu$ is compact if and only if
 \[
  m(t) = \inf_{x \in \supp\mu}\mu(B(x,t)) > 0 \qquad \text{for all }t > 0.
 \]
\end{lemma}
\begin{proof}
Let us first assume compactness of $\supp\mu$. Then the claim follows from the lower semicontinuity of the function $x \mapsto \mu(B(x,t))$.
 Let us still provide a short proof for the convenience of the reader. 
 
 Let $t>0$.
 If $m(t) = 0$ would hold, then there would exist a sequence $(x_i)_{i=1}^\infty$ in $\supp\mu$
 such that $\mu(B(x_i,t)) \to 0$. By compactness of $\supp\mu$ we may assume that $x_i \to x \in \supp\mu$.
 Since 
 \[
  B(x,t) = \bigcup_{i=1}^\infty B(x_i,t-\sfd(x_i,x)),
 \]
 we have
 \[
  \mu(B(x,t)) \le \liminf_{i \to \infty} \mu(B(x_i,t)) = 0.
 \]
 This contradicts the fact that, by the definition of the support, $\mu(B(x,t))>0$. 

Now let us show the converse direction and assume $m(t)>0$ for all $t>0$. Then we can prove that $\supp\mu$ is totally bounded:
Given $r>0$ we can choose a maximal collection of disjoint balls $B(x,r/2)$, $x\in\supp\mu$.
This collection is finite, since it contains at most $1/m(r/2)$ balls.
Doubling the radius of the balls in this collection gives a finite cover of $\supp\mu$ using balls of radius $r$.
\end{proof}

Next we prove the key lemma of the paper.
For $\delta > 0$, by a $\delta$-connected set $A \subset X$ we mean that for all $x,y \in A$ there exists a sequence $(x_i)_{i=1}^N$ in $A$
such that $x_i=x$, $x_N = y$ and $\sfd(x_i,x_{i+1}) \le \delta$ for all $i$.
by $0$-connected set we simply mean a connected set.
For proving Theorem \ref{thm:main}, we will use Lemma \ref{lma:bigpartlongway} with $\delta=\eps=0$. Positive
$\delta$ and $\eps$ will appear later in the proof of Theorem \ref{thm:convergence}.

\begin{lemma}\label{lma:bigpartlongway}
 Let $(X,\sfd)$ be a complete metric space,
 let $\mu \in \prob{X}$ with $\delta$-connected support for some $\delta \ge 0$, and suppose that
 \[
  m(t) := \inf_{x \in \supp\mu}\mu(B(x,t))> 0 \qquad\text{for all }t>0.
 \] 
 Furthermore, let $\varepsilon\ge0$ and $\nu \in \prob{B(\supp\mu,\varepsilon)}$ be such that there exists $\lambda \in \Pi(\mu,\nu)$ with the property that
 \begin{equation}\label{eq:assumption}
    \lambda\big(\{(x,y) \in X \times X\,:\,\sfd(x,y) \ge r\}\big) < \frac{m(r)}{2}
 \end{equation}
 for some $r > 0$. Then $W_\infty(\mu,\nu)\le 17r + 4\varepsilon + \delta$.
\end{lemma}
\begin{proof}
 Let $\{x_i\}_{i=1}^N$ be a maximal $4r$-separated net of points in $\supp \mu$. 
 The fact that $N < \infty$ follows from the assumption $m(r) > 0$.
 With the points $\{x_i\}_{i=1}^N$ we define
 a Borel partition $\{U_i\}_{i = 1}^N$ of $B(\supp \mu,\varepsilon)$ by inductively setting
 \[
  U_i := \big\{x \in B(\supp\mu,\varepsilon)\,:\,\sfd(x,x_i)\le \sfd(x,x_j) \text{ for all }j \ne i\big\} \setminus \bigcup_{k < i}U_k.
 \]
 Notice that by the $4r$-separation of $\{x_i\}_{i=1}^N$ we have
 \[
  B(x_i,2r) \cap B(\supp\mu,\epsilon) \subset U_i \qquad\text{for all }i.
 \]
 Since $\{x_i\}_{i=1}^N$ is a maximal $4r$-separated net of points
 in $\supp \mu$ and $U_i \subset B(\supp\mu,\varepsilon)$, we have
 \begin{equation}\label{eq:Udiam}
  \diam(U_i) \le 8r + 2\varepsilon \qquad\text{for all }i.
 \end{equation}
 For all $x \in B(x_i,r)$ and $y \notin U_i$ we have $\sfd(x,y)>r$ by the triangle inequality. Thus
 from $B(x_i,r) \ge m(r)$ and \eqref{eq:assumption} we have that
 \begin{equation}\label{eq:lowerbound}
   \lambda(U_i \times U_i) \ge \frac{m(r)}{2} \qquad\text{for all }i.
 \end{equation}
 In particular, $\nu(U_i)>0$ for all $i$.
 
 Let us now define a new transport $\eta \in \Pi(\mu,\nu)$ in three parts. The first one takes care of the small distance transports
 between different $U_i$, the second one takes care of the long distance transports and the third one handles the remaining transports
 inside the sets $U_i$. Let us write
 \[
  A := \{(x,y) \in X \times X\,:\,\sfd(x,y) \ge r\}.
 \]
 For each pair $i,j \in \{1,\dots,N\}$, $i \ne j$ define
 \[
  \eta_{i,j} := \lambda((U_i\times U_j) \setminus A)\left(\frac{\mu\restr{U_i}}{\mu(U_i)} \times \frac{\nu\restr{U_j}}{\nu(U_j)}\right).
 \]
 Notice that if $(U_i\times U_j) \setminus A \ne \emptyset$ we have by \eqref{eq:Udiam} for all $(x,y) \in U_i\times U_j$ the estimate
 \begin{equation}\label{eq:dist1}
   \sfd(x,y) \le \diam(U_i) + \diam(U_j) + r \le 8r + 2\varepsilon + 8r  + 2\varepsilon + r = 17r + 4\varepsilon.
 \end{equation}
 The measures $\eta_{i,j}$ form the first part of the new transport.
 
 For the second part of the transport we select for each $i,j \in \{1,\dots,N\}$ a suitable
 chain of sets joining $U_i$ to $U_j$. In the case $\delta>0$ such chain is obtained as follows.
 Take $x \in U_i\cap\supp\mu$ and $y \in U_j\cap \supp\mu$. By the $\delta$-connectedness of $\supp\mu$ 
 there exists a sequence $(y_l)_{l=1}^M$ in $\supp\mu$
 such that $y_1 = x$, $y_M = y$ and $\sfd(y_l,y_{l+1})\le\delta$ for all $l$.
 Define $k_1 = i$ and inductively for $l>1$ the number $k_l$ as the one satisfying $y_{p+1} \in U_{k_l}$ where
 $p$ is the largest index in $\{1, \dots, M\}$ such that $y_p \in U_{k_{l-1}}$.
 This way we have defined a sequence $(k_l)_{l=1}^L \subset \{1,\dots,N\}$
 having the properties $\dist(U_{k_l},U_{k_{l+1}}) \le \delta$ for all $l \in \{1,\dots, L-1\}$, $U_{k_1} = U_i$, $U_{k_L} = U_j$ and
 $U_{k_l} \ne U_{k_l'}$ for $l \ne l'$. In the case $\delta = 0$ such sequence exists by a similar argument.

 Now we define
 \[
  \tilde\eta_{i,j} := \sum_{l = 1}^{L-1}
  \lambda((U_i\times U_j) \cap A)\left(\frac{\mu\restr{U_{k_l}}}{\mu(U_{k_l})} \times \frac{\nu\restr{U_{k_{l+1}}}}{\nu(U_{k_{l+1}})}\right).
 \]
 For all $(x,y) \in \supp(\tilde\eta_{i,j})$ we have $i \ne j$ and thus
 there exists some $l$ such that $x \in U_{k_l}$ and $y \in U_{k_{l+1}}$. Therefore, using \eqref{eq:Udiam} we get
 \begin{equation}\label{eq:dist2}
   \sfd(x,y) \le \diam(U_{k_l}) + \diam(U_{k_{l+1}}) + \dist(U_{k_l},U_{k_{l+1}}) \le 8r  + 2\varepsilon+ 8r  + 2\varepsilon+ \delta = 16r + 4\varepsilon + \delta.
 \end{equation}
 Notice that
 \[
  \tilde\eta_{i,j}(X \times U_k) = \begin{cases}
                                    \lambda((U_i\times U_j) \cap A), & \text{if }k \in \{k_2,\dots, k_L\},\\
                                    0, & \text{otherwise}
                                   \end{cases}
 \]
 and
 \[
  \tilde\eta_{i,j}(U_k\times X) = \begin{cases}
                                    \lambda((U_i\times U_j) \cap A), & \text{if }k \in \{k_1,\dots, k_{L-1}\},\\
                                    0, & \text{otherwise}.
                                   \end{cases}
 \]
 Therefore
 \begin{equation}\label{eq:cancellation}
  \tilde\eta_{i,j}(X \times U_k) - \tilde\eta_{i,j}(U_k\times X)
    = \begin{cases}
         \lambda((U_i\times U_j) \cap A), & \text{if }k = j,\\
         -\lambda((U_i\times U_j) \cap A), & \text{if }k = i,\\
            0, & \text{otherwise}
     \end{cases}
 \end{equation}
 and
 \begin{equation}\label{eq:tildesum}
  \sum_{j,k} \tilde\eta_{j,k}(U_i \times X) \le \sum_{j,k} \lambda((U_j\times U_k) \cap A) \le \lambda(A).
 \end{equation}

 The remaining third part will be given by the measures
 \[
  \eta_{i,i} := \beta_i\left(\frac{\mu\restr{U_i}}{\mu(U_i)} \times \frac{\nu\restr{U_i}}{\nu(U_i)}\right),
 \]
 where
 \[
  \beta_i := \mu(U_i) - \sum_{j \ne i}\lambda((U_i\times U_j) \setminus A) - \sum_{j,k} \tilde\eta_{j,k}(U_i \times X)
  \ge \lambda(U_i\times U_i) - \lambda(A) > 0,
 \]
 by \eqref{eq:tildesum}, the assumption \eqref{eq:assumption} and the estimate \eqref{eq:lowerbound}.
 Clearly, for all $x,y \in U_i$ we have by \eqref{eq:Udiam} that
 \begin{equation}\label{eq:dist3}
  \sfd(x,y) \le \diam(U_i) \le 8r + 2\varepsilon.
 \end{equation}
 Let us now write $\eta = \sum_{i,j}(\eta_{i,j} + \tilde\eta_{i,j})$.
 Denoting by $\mathtt{p}_i$ the projection to $i$:th component we have
 \begin{align*}
  \mathtt{p_1}_\sharp\eta  = \sum_{i,j}(\mathtt{p_1}_\sharp\eta_{i,j} + \mathtt{p_1}_\sharp\tilde\eta_{i,j})
  & = \sum_{i}\left(\beta_i + \sum_{j\ne i}\lambda((U_i\times U_j) \setminus A) + \sum_{j,k} \tilde\eta_{j,k}(U_i \times X) \right)\frac{\mu\restr{U_i}}{\mu(U_i)}\\
  & = \sum_{i}\mu(U_i)\frac{\mu\restr{U_i}}{\mu(U_i)} = \mu
 \end{align*}
 and, using \eqref{eq:cancellation}, we get
 \begin{align*}
  \mathtt{p_2}_\sharp\eta  = \sum_{i,j}(\mathtt{p_2}_\sharp\eta_{i,j} + \mathtt{p_2}_\sharp\tilde\eta_{i,j})
  & = \sum_{j}\left(\beta_j + \sum_{i\ne j}\lambda((U_i\times U_j) \setminus A) + \sum_{i,k} \tilde\eta_{i,k}(X \times U_j) \right)\frac{\nu\restr{U_j}}{\nu(U_j)}\\
  & = \sum_{j}\Bigg(\mu(U_j) + \sum_{i\ne j}\left(\lambda((U_i\times U_j) \setminus A) - \lambda((U_j\times U_i) \setminus A)\right)\\
  & \qquad + \sum_{i,k} \left(\tilde\eta_{i,k}(X \times U_j) - \tilde\eta_{i,k}(U_j \times X)\right) \Bigg)\frac{\nu\restr{U_j}}{\nu(U_j)}\\
  & = \sum_{j}\Bigg(\mu(U_j) + \sum_{i\ne j}\left(\lambda((U_i\times U_j) \setminus A) - \lambda((U_j\times U_i) \setminus A)\right)\\
  & \qquad + \sum_{i\ne j} \left(\lambda((U_i\times U_j) \cap A) - \lambda((U_j\times U_i) \cap A)\right) \Bigg)\frac{\nu\restr{U_j}}{\nu(U_j)}\\
  & = \sum_{j}\left(\mu(U_j) + \lambda((X\setminus U_j)\times U_j) - \lambda(U_j \times (X \setminus U_j)) \right)\frac{\nu\restr{U_j}}{\nu(U_j)}\\
  & = \sum_{j}\nu(U_j)\frac{\nu\restr{U_j}}{\nu(U_j)} = \nu.
 \end{align*}
 Thus $\eta \in \Pi(\mu,\nu)$.  

 Therefore, by the estimates \eqref{eq:dist1}, \eqref{eq:dist2}, and \eqref{eq:dist3} we have
 \[
  W_\infty(\mu,\nu) \le \eta\text{-}\esup_{(x,y) \in X^2} \sfd(x,y)\le 17r + 4\varepsilon + \delta
 \]
 as claimed.
\end{proof}

\subsection{Existence of $W_\infty$ lower bounds}

The following lemma will be used in showing the necessity of compactness and connectedness of $\supp \mu$ in Theorem
\ref{thm:main}.

\begin{lemma}\label{lma:komponentit}
Let $(X,\sfd)$ be a complete metric space and $\mu \in \prob{X}$.
If $\mu$ satisfies \eqref{eq:maineq}, then there cannot exist nonempty Borel sets $A,B\subset \supp\mu$ such that 
$A\cup B=\supp\mu$ and $\dist(A,B)>0$.
\end{lemma}

\begin{proof}
Suppose such $A$ and $B$ exist. Since $A$ and $B$ are nonempty and $\dist(A,B)>0$, we have $\mu(A)>0$ and $\mu(B)>0$.
Take $x \in X$ and let $R>0$ be large enough so that
$\mu(\tilde A)>0$ and $\mu(\tilde B)>0$ where 
$\tilde A := A \cap B(x,R)$ and $\tilde B := B \cap B(x,R)$.
Define for all $0<t<\mu(\tilde A)$
\[
\nu_t:=\mu\restr{X \setminus B(x,R)}+\frac{\mu(\tilde A)-t}{\mu(\tilde A)}\mu\restr{\tilde A}+\frac{\mu(\tilde B)+t}{\mu(\tilde B)}\mu\restr{\tilde B}.
\]
Now, 
using
\[
 \tilde\lambda = (\id,\id)_\sharp\mu\restr{X \setminus B(x,R)} + \frac{1}{\mu(B(x,R))}\mu\restr{B(x,R)}\times\nu_t\restr{B(x,R)} \in \Pi(\mu,\nu_t),
\]
we have
\[
\inf_{\lambda\in\Pi(\mu,\nu_t)} \int h\circ \sfd \,\d\lambda  \le \int h\circ \sfd \,\d\tilde\lambda \leq th(2R).
\]
On the other hand, since $\nu_t(A) < \mu(A)$, we also have $W_\infty(\mu,\nu_t)\geq \dist(A,B)>0$. Since $\nu_t\in\prob{\supp\mu}$ we can apply \eqref{eq:maineq} and thus
\[
th(2R)\geq \inf_{\lambda\in\Pi(\mu,\nu_t)}\int h\circ \sfd\,\d\lambda \; \geq \; \omega\big(W_\infty(\mu,\nu_t)\big) \geq \omega(\dist(A,B))>0.
\]
Letting $t\to 0$ gives a contradiction.
\end{proof}

\begin{proof}[Proof of Theorem \ref{thm:main}]
 Let us first show that connectedness and compactness of $\supp\mu$ imply \eqref{eq:maineq}.
 We may suppose $W_\infty(\mu,\nu) > 0$. 
 Now applying Lemma \ref{lma:bigpartlongway} with $\varepsilon=\delta = 0$ we see that for all $\lambda \in \Pi(\mu,\nu)$ we have
 \[
  \lambda\left\{(x,y) \in X \times X\,:\,\sfd(x,y) \ge \frac{W_\infty(\mu,\nu)}{17}\right\} \ge \frac{m(\frac{W_\infty(\mu,\nu)}{17})}{2}.
 \]
 In particular,
 \[
  \int h\circ \sfd\,\d\lambda \ge \frac{m(\frac{W_\infty(\mu,\nu)}{17})}{2}h(\frac{W_\infty(\mu,\nu)}{17}).
 \]
 By Lemma \ref{lma:mexists} the function $m$ is positive. Thus the inequality \eqref{eq:maineq} holds with the $\omega$
 claimed in the theorem.
\medskip
 
Now we prove that if $\mu$ satisfies \eqref{eq:maineq}, then $\supp\mu$ is compact and connected.
Let us first show that $\supp \mu$ is compact. By Lemma \ref{lma:mexists} it suffices to show that
 $m(r)>0$ for all $r>0$. 
 To this end fix $r>0$ and $x\in\supp\mu$. 
 Let $B:=B(x,r)$, $A_1:=\supp\mu\cap B(x,2r)\setminus B$ and $A_2:=\supp\mu\setminus B(x,2r)$. We may assume $\mu(B)<1$. Then $\mu(A_1)>0$ by Lemma \ref{lma:komponentit}. Define 
\[
\nu_t:=\mu\restr{A_2}+\frac{\mu(A_1)-t}{\mu(A_1)} \mu\restr{A_1}+(\mu(B)+t)\delta_x, \quad\textrm{for any}\;\: 0<t<\mu(A_1).
\]
Now, 
using
\[
 \tilde\lambda = (\id,\id)_\sharp\left(\mu\restr{A_2}+\frac{\mu(A_1)-t}{\mu(A_1)} \mu\restr{A_1}\right)
  + \left(\frac{t}{\mu(A_1)}\mu\restr{A_1} + \mu\restr{B}\right)\times\delta_x,
\]
we have
\[
\inf_{\lambda\in\Pi(\mu,\nu_t)} \int h\circ \sfd \,\d\lambda \leq th(2r)+\mu(B)h(r),
\]
and on the other hand we also have $W_\infty(\mu,\nu_t)\geq r$. 
Since $\nu\in\prob{\supp\mu}$ we obtain by applying \eqref{eq:maineq}  the estimate
\[
th(2r)+\mu(B)h(r) \geq \omega(r),
\]
which gives $\mu(B)\geq \omega(r)/h(r)>0$ by letting $t\to 0$.
Thus $m(r)>0$ and by Lemma \ref{lma:mexists} $\supp\mu$ is compact.
Since $\supp \mu$ is compact, if it had two connected components they
would have positive distance from each other. This would contradict
Lemma \ref{lma:komponentit}.
Thus $\supp\mu$ is connected.
\end{proof}

For the sharpness of Theorem \ref{thm:main} we have the following result.

\begin{proposition}\label{prop:sharpness}
 Let $\mu\in \prob{X}$ be with compact and connected support.
 Then for all $0 < r < \diam(\supp \mu)$ there exists $\nu \in \prob{\supp \mu}$ such that $W_\infty(\mu,\nu) = r$ and
  \[
  \inf_{\lambda \in \Pi(\mu,\nu)}\int h\circ \sfd\,\d\lambda \le \tilde\omega(W_\infty(\mu,\nu))
 \]
 with $\tilde\omega(t) = m(t)h(t)$, where $m(t) = \inf_{x \in \supp\mu}\mu(B(x,t))$.
\end{proposition}
\begin{proof}
 Let $x \in \supp\mu$ be such that $\mu(B(x,r)) = m(r)$. Such $x$ exists by the compactness of $\supp\mu$
 and the lower semicontinuity of the function $x \to \mu(B(x,r))$. Define
 \[
  \nu = \mu\restr{X \setminus B(x,r)} + \mu(B(x,r))\delta_x.
 \]
 Then by the connectedness of $\supp \mu$ we have $W_\infty(\mu,\nu) = r$. Clearly
 \begin{align*}
  \inf_{\lambda \in \Pi(\mu,\nu)}\int h\circ \sfd\,\d\lambda & 
  \le \mu(B(x,r))\inf_{\lambda \in \Pi((\mu(B(x,r)))^{-1}\mu\restr{B(x,r)},\delta_x)}\int h\circ \sfd\,\d\lambda\\
  & = \mu(B(x,r))\int_{B(x,r)} h\circ \sfd(z,x)\,\d\mu(z)
  \le \mu(B(x,r))h(r) = m(r)h(r).
 \end{align*}
\end{proof}

Notice that Proposition \ref{prop:sharpness} does not in general give a sharp bound since the inequality
\[
\int_{B(x,r)} h\circ \sfd(z,x)\,\d\mu(z)\le h(r)
\]
could be sharpened.

\subsection{Comparison of convergence in $W_p$ and $W_\infty$}

Let us then turn to the proof of Theorem \ref{thm:convergence}. We start with a simple observation.

\begin{lemma}\label{lma:HausdoffWinfty}
 Let $\mu,\nu \in \prob{X}$. Then
 \[
  \sfd_H(\supp\mu,\supp\nu) \le W_\infty(\mu,\nu).
 \]
\end{lemma}
\begin{proof}
 Let $x \in \supp\mu$. Then for all $\varepsilon>0$ we have $\mu(B(x,\varepsilon))>0$ and hence
 \[
  W_\infty(\mu,\nu) \ge \dist(B(x,\varepsilon),\supp\nu).
 \]
 Taking $\varepsilon \to 0$, supremum over $x \in \supp\mu$ and making a similar argument for $x \in \supp\nu$,
 the claim follows.
\end{proof}

Due to compactness, under weak convergence of measures, Hausdorff convergence of supports is the same as uniform 
lower bounds on the measure of balls. This is the content of the next lemma.

\begin{lemma}\label{lma:condition2equiv}
 Let $(X,\sfd)$ be a complete metric space.
 Suppose $(\mu_i)_{i=1}^\infty$ is a sequence of compactly supported probability measures in $X$
 and $\mu \in \prob{X}$ such that $W_p(\mu_i,\mu) \to 0$ as $i \to \infty$.
 Then
 \begin{equation}\label{eq:uniformbound}
  m(r) := \inf_{i \in \N}\inf_{x \in \supp \mu_i}\mu_i(B(x,r)) > 0\qquad \text{for all }r>0 
 \end{equation}
 if and only if $\sfd_{H}(\supp \mu_i,\supp \mu) \to 0$ as $i \to \infty$.
\end{lemma}
\begin{proof}
 Let us first assume that $\sfd_{H}(\supp \mu_i,\supp \mu) \to 0$ as $i \to \infty$.
 Since $\supp \mu_i$ are compact and since compactness is preserved in Hausdorff convergence, also $\supp \mu$ is compact.
 Thus by Lemma \ref{lma:mexists} it holds $\inf_{x \in \supp\mu}\mu(B(x,r))>0$ for all $r>0$. 

Fix $r>0$. Then by assumption there exists $i_1\in\N$ such that $\sfd_{H}(\supp \mu_i,\supp \mu)<{r}/{4}$ for all $i\ge i_1$. 
Since $\supp\mu$ is compact, it can be covered with a finite number of balls $B(x,{r}/{4})$, $r\in\supp\mu$. Let $\{B(x_j,{r}/{4})\}_{j=1}^N$ be this finite collection.
Note that $\{B(x_j,{r}/{2})\}_{j=1}^N$ also covers $\supp\mu_i$ for any $i\ge i_1$. 

The convergence $W_p(\mu_i,\mu)\to 0$ implies that for every $j \in \{1,\ldots, N\}$ we have
\[
\liminf_{i\to\infty} \mu_i(B(x_j,\frac{r}{2})) \ge \mu(B(x_j,\frac{r}{2})) \ge \inf_{x \in \supp\mu}\mu(B(x,\frac{r}{2}))>0.
\]
Thus there exists $i_2\in\N$ so that if $i\ge i_2$, the inequality $\mu_i(B(x_j,{r}/{2})) \ge \frac12 \inf_{x \in \supp\mu}\mu(B(x,{r}/{2}))$ holds for every $j=1,\ldots, N$. 

We take $i_0=\max\{i_1,i_2\}$. Now, if $i\ge i_0$ and $x\in\supp\mu_i$, then, since $i\ge i_1$, there exists $x_j\in\supp\mu$ such that $B(x_j,\frac{r}{2})\subset B(x,r)$. Combining this with $i\ge i_2$ we obtain
\[
\mu_i(B(x,r)) \ge \mu_i(B(x_j,\frac{r}{2})) \ge \frac12 \inf_{x \in \supp\mu}\mu(B(x,\frac{r}{2})).
\]
We conclude that \eqref{eq:uniformbound} holds, since we have
\[
m(r) \ge \min\big\{ \min_{1 \le i < i_0}\inf_{x \in \supp \mu_i}\mu_i(B(x,r)), \frac12 \inf_{x \in \supp\mu}\mu(B(x,\frac{r}{2})) \big\} > 0,
\]
where $\inf_{x \in \supp \mu_i}\mu_i(B(x,r))>0$ follows from compactness of $\supp\mu_i$ and Lemma \ref{lma:mexists}.
 \medskip
 
 Let us then assume that \eqref{eq:uniformbound} holds and show that $\sfd_{H}(\supp \mu_i,\supp \mu) \to 0$ as $i \to \infty$.
 First take $x \in \supp\mu$. Then
 \begin{equation}\label{eq:Wpestimate}
  W_p^p(\mu_i,\mu) \ge \left(\frac{\dist(x,\supp\mu_i)}{2}\right)^p\mu\left(B\left(x,\frac{\dist(x,\supp\mu_i)}{2}\right)\right).
 \end{equation} 
 Since $W_p^p(\mu_i,\mu) \to 0$, also $\dist(x,\supp\mu_i) \to 0$.
 Now by the uniform bound \eqref{eq:uniformbound} we also have $\mu(B(x,r))\ge m(r)$ for all $r>0$.
 Thus the inequality \eqref{eq:Wpestimate} leads to the uniform estimate
 \[
  W_p^p(\mu_i,\mu) \ge \left(\frac{\dist(x,\supp\mu_i)}{2}\right)^p m\left(\frac{\dist(x,\supp\mu_i)}{2}\right).
 \]
 Making a similar estimate for $x \in \supp\mu_i$ we obtain
 \[
  W_p^p(\mu_i,\mu) \ge \left(\frac{\sfd_H(\supp\mu_i,\supp\mu)}{2}\right)^p m\left(\frac{\sfd_H(\supp\mu_i,\supp\mu)}{2}\right).
 \]
 This gives $\sfd_H(\supp\mu_i,\supp\mu) \to 0$ as $i \to \infty$.
\end{proof}

Now we can apply Lemma \ref{lma:bigpartlongway} to prove Theorem \ref{thm:convergence}.

\begin{proof}[Proof of Theorem \ref{thm:convergence}]
 Assume first that $W_\infty(\mu_i,\mu) \to 0$ as $i \to \infty$. 
 Then, by \eqref{eq:trivial}, also $W_p(\mu_i,\mu) \to 0$ as $i \to \infty$.
 By Lemma \ref{lma:HausdoffWinfty} we have $\sfd_H(\supp \mu_i,\supp\mu) \to 0$ as $i \to \infty$.
 In order to see that (3) holds, suppose there exist sequences of positive measures $(\mu_i^1)_{i=1}^\infty$ and $(\mu_i^2)_{i=1}^\infty$ such that
 $\mu_i = \mu_i^1 + \mu_i^2$ for all $i$, $W_p(\mu_i^1,\mu^1) \to 0$ and $W_p(\mu_i^2,\mu^2) \to 0$
 for some measures $\mu^1$ and $\mu^2$, and with $d := \inf_{i\in\N}\dist(\supp\mu_i^1,\supp\mu_i^2)>0$.
 Then, if $\mu^1(X) \ne \mu_i^1(X)$ we have that any transport from $\mu_i$ to $\mu$ must transport measure between $\mu_i^1$
 and $\mu^2$. Thus
 \[
  W_\infty(\mu,\mu_i) \ge d - \sfd_H(\supp\mu_i,\supp\mu).
 \]
 Hence by Lemma \ref{lma:HausdoffWinfty}
 \[
  W_\infty(\mu,\mu_i) \ge \frac{d}{2}
 \]
 and by $W_\infty(\mu_i,\mu) \to 0$ there exists $i_0 \in \N$ such that $\mu_i^1(X) = \mu^1(X)$ for all $i \ge i_0$.
 \medskip
 
 Let us then show the converse direction. We will again apply Lemma \ref{lma:bigpartlongway}.
 Take $s>0$. Since $\sfd_H(\supp \mu_i,\supp\mu) \to 0$ as $i \to \infty$, also $\supp \mu$ is compact.
 Then by Lemma \ref{lma:mexists}
 \[
  m(t) := \inf_{x \in \supp\mu}\mu(B(x,t)) > 0 \qquad \text{for all }t > 0.
 \]
 By compactness $\supp \mu$ consists of a finite number of $s$-connected components. 
 We denote the corresponding parts of the measures by $(\mu^k)_{k=1}^\infty$. 
 Notice that by the Hausdorff-convergence of the supports we have $\mu_i \in \prob{B(\supp\mu,\sfd_H(\supp \mu_i,\supp\mu))}$.
 By this and the assumption (3) there exists $i_0 \in \N$ such that 
 $\mu_i^k(X) = \mu^k(X)$ for all $k$ and $i \ge i_0$ and
 $\sfd_H(\supp \mu_i,\supp\mu) < s$ for all $i \ge i_0$.
 
 Then, similarly as in the proof of Theorem \ref{thm:main}, by this time considering each component $\mu^k$ separately,
 we get from Lemma \ref{lma:bigpartlongway},
 with $\varepsilon = \delta = s$, the estimate
 \[
  W_p^p(\mu_i,\mu) \ge \frac{m(\frac{W_\infty(\mu_i,\mu)-5s}{17})}{2}\left(\frac{W_\infty(\mu_i,\mu)-5s}{17}\right)^p.
 \]
 Since $m$ is nondecrasing and positive and $W_p(\mu_i,\mu) \to 0$ as $i \to \infty$, by taking $s \to 0$ we have
 also $W_\infty(\mu_i,\mu) \to 0$ as $i \to \infty$.
\end{proof}

Let us end this section with an example, mentioned in the Introduction, that shows the necessity of the compactness assumption
in Theorem \ref{thm:convergence}.

\begin{example}\label{ex:noncompact}
 Let us first define $\mu \in \prob{\R}$ by setting
 \[
  \mu := c\sum_{n=1}^\infty e^{-n}\mathcal{L}\restr{[-n(n+1),-(n-1)n]\cup[(n-1)n,n(n+1)]},
 \]
 where the constant $c$ is chosen so that the total mass is one.
 For each $i$ define $\mu_i \in \prob{\R}$ as
 \begin{align*}
  \mu_i := & c\sum_{n=1}^i e^{-n}\mathcal{L}\restr{[-n(n+1),-(n-1)n]\cup[(n-1)n,n(n+1)]}\\
   & + \frac{c}2\sum_{n=i+1}^\infty e^{-n}\left(\mathcal{L}\restr{[-n(n+1),-(n-1)n]\cup[(n-1)n,n(n+1)]} + 2n\delta_{-n^2} + 2n\delta_{n^2}\right).
 \end{align*}
 Then $\supp \mu_i = \supp \mu = \R$, and thus in particular $\sfd_H(\supp\mu_1,\supp\mu) = 0$.
 Moreover, for every $1 < p < \infty$
 \[
  W_p^p(\mu_i,\mu) \le \frac{c}{2}\sum_{n=i+1}^\infty 4e^{-n}n^{p+1} \to 0 \qquad \text{as }i \to \infty.
 \]
 However, we have
 \[
  W_\infty(\mu_i,\mu) = \infty \qquad \text{for all }i \in \N,
 \]
 since for all $n > i$ it holds
 \[
  \mu([n^2-\frac{n}4,n^2+\frac{n}4]) = \frac{c ne^{-n}}2 < cne^{-n} = \mu_i(\{n^2\}).
 \]

\end{example}

\section{$L^\infty$ estimate for optimal transport plans}

This section is devoted to $L^\infty$ estimates on the level of optimal transport plans.
We present the proof of Theorem \ref{thm:planestimate} and after that, in Proposition \ref{prop:nicelyconnectedsupport},
we provide a sample case where the conditions of Theorem \ref{thm:planestimate} are satisfied.


%

Now we are dealing with optimal transport plans and their properties. To help with this we recall that in our case optimality of a transport plan is equivalent with cyclical monotonicity.

\begin{definition}
Let $c\colon X\times X\to \R$ be a continuous function. A transport plan $\lambda\in\prob{X\times X}$ is said to be \emph{$c$-cyclically monotone}, if for every finite number of points $(x_1,y_1),\ldots,(x_K,y_K)\in\supp\lambda$ we have
\[
\sum_{i=1}^K c(x_i,y_i) \leq \sum_{i=1}^K c(x_i,y_{\sigma(i)})
\]
for any permutation $\sigma$ of the set $\{1,2,\ldots,K\}$.
\end{definition}

The connection between optimality and cyclical monotonicity is proven for example in Villani's book \cite{Villani09}. The following Theorem is a special case of Theorem 5.10 in \cite{Villani09}.

\begin{theorem}\label{thm:optimaliscyclical}
Let $\mu,\nu\in\prob{X}$ be compactly supported and let $c\colon X\times X\to [0,\infty)$ be a continuous cost function. Then a transport plan $\lambda\in\Pi(\mu,\nu)$ is optimal if and only if it is $c$-cyclically monotone.
\end{theorem}

To prove Theorem \ref{thm:planestimate} we first prove that in our situation we can get a uniform version of \eqref{eq:eisykli}.

\begin{lemma}\label{lma:rhoexists}
Suppose that $(X,\sfd)$, $h$ and $\mu$ are such that (1) of Theorem \ref{thm:planestimate} holds. Then there exists a nondecreasing function $\rho \colon [0,\infty)\to[0,\infty)$ with $\rho(t)>0$ for all $t>0$, which satisfies the following condition: For every $x,y\in\supp\mu$, $x\neq y$, there exist $N\in\N$ and a sequence $(z_i)_{i=0}^{N+1}\subset\supp\mu$ such that $z_0=x$, $z_{N+1}=y$ and
\begin{equation}\label{eq:eisyklir}
\sum_{i=0}^N h\big(\sfd(z_i,z_{i+1})+\rho(\sfd(x,y))\big) \;<\; h(\sfd(x,y)).
\end{equation}
\end{lemma}

\begin{proof}
Fix $0<R\leq \diam(\supp\mu)$. Define the set 
\[
A_R:=\{(x,y)\in\supp\mu\times\supp\mu: \sfd(x,y)\geq R\}.
\]
Since $\supp\mu$ is compact, the set $A_R$ is compact as well. In addition, condition \eqref{eq:eisykli} states that for every $(x,y)\in A_R$ there exist $N_{(x,y)}\in\N$ and a sequence $(z_i^{(x,y)})_{i=1}^{N_{(x,y)}}\subset\supp\mu$ such that
\[
\varepsilon_{(x,y)}:=h(\sfd(x,y))-\big(h(\sfd(x,z_1^{(x,y)}))+\sum_{i=1}^{N_{(x,y)}-1}h(\sfd(z_i^{(x,y)},z_{i+1}^{(x,y)}))+h(\sfd(z_{N_{(x,y)}}^{(x,y)},y))\big)>0.
\]
Now, given $(x,y)\in A_R$ we can use the continuity of $h$ to get a radius $r_{(x,y)}>0$ such that for any $(x',y')\in B(x,r_{(x,y)})\times B(y,r_{(x,y)})$ we still have
\begin{equation}\label{eq:peiteey}
h(\sfd(x',y'))-\big(h(\sfd(x',z_1^{(x,y)}))+\sum_{i=1}^{N_{(x,y)}-1}h(\sfd(z_i^{(x,y)},z_{i+1}^{(x,y)}))+h(\sfd(z_{N_{(x,y)}}^{(x,y)},y'))\big)>\frac{\varepsilon_{(x,y)}}{2}.
\end{equation}
We can cover $A_R$ with open sets $U_{(x,y)}:=B(x,r_{(x,y)})\times B(y,r_{(x,y)})$, $(x,y)\in A_R$, 
and thanks to compactness of $A_R$ a finite number of these sets is enough to cover $A_R$.
Let us denote the sets in this finite cover by $U_j$, $j \in \{1,\ldots, M\}$. We also denote by $\varepsilon_j$ the $\varepsilon_{(x,y)}$ corresponding to $U_j$, and similarly we denote $N_j$ and $(z_i^j)_{i=1}^{N_j}$.

Define 
\[
\varepsilon:=\min_{1\leq j \leq M}\frac{\varepsilon_j}{2} \qquad \textrm{and}\qquad N:=\max_{1\leq j\leq M} N_j.
\]
Then, since $h$ is uniformly continuous on the interval $[0,2\diam(\supp\mu)]$, there exists $\delta>0$ such that 
\begin{equation}\label{eq:tasjva}
h(s+\delta)<h(s)+\frac{\epsilon}{N+1}  \quad \textrm{for any } s\in [0,\diam(\supp\mu)].
\end{equation}

Next we prove that if $(x,y)\in A_R$, then \eqref{eq:eisyklir} holds with $\rho(\sfd(x,y))$ replaced by $\delta$. So fix $(x,y)\in A_R$. Then we have $(x,y)\in U_j$ for some $j$. We take $(z_i^j)_{i=1}^{N_j}$ as the sequence required in \eqref{eq:eisyklir}. Using \eqref{eq:tasjva} and then \eqref{eq:peiteey} we obtain
\begin{align*}
 h\big(\sfd(x,z_1^j)+\delta\big) & +\sum_{i=1}^{N_j-1}h\big(\sfd(z_i^j,z_{i+1}^j)+\delta\big)+h\big(\sfd(z_{N_j}^j,y)+\delta\big) \\
<\; & h\big(\sfd(x,z_1^j)\big)+\frac{\varepsilon}{N+1} +\sum_{i=1}^{N_j-1}\Big( h\big(\sfd(z_i^j,z_{i+1}^j)\big)+\frac{\epsilon}{N+1}\Big)+h\big(\sfd(z_{N_j}^j,y)\big)+\frac{\epsilon}{N+1} \\
\leq\; & h\big(\sfd(x,z_1^j)\big)+\sum_{i=1}^{N_j-1}h\big(\sfd(z_i^j,z_{i+1}^j)\big)+h\big(\sfd(z_{N_j}^j,y)\big)+\varepsilon \\
<\; & h(\sfd(x,y))-\frac{\varepsilon_j}{2}+\varepsilon \;\leq\; h(\sfd(x,y)),
\end{align*}
which means that \eqref{eq:eisyklir} holds with $\delta$.

Now we could simply define $\rho(R):=\delta$, but to make sure $\rho$ is nondecreasing we will do the following: Take a decreasing sequence $(R_k)_{k=1}^\infty$, $R_k\to 0$ (for example $R_k=2^{-k}\diam(\supp\mu)$). Then we use the above argument for $R=R_k$ to obtain $\delta_k$. With these we can define $\rho$: First we set $\rho(t):=\delta_1$ for $t\geq R_1$ and then we define $\rho$ for intervals $[R_{k+1},R_k)$ recursively by setting $\rho(t):=\min\{\delta_{k+1},\rho(R_k)\}$ for $t\in[R_{k+1},R_k)$.

This definition ensures that $\rho$ is nondecreasing and strictly positive in $(0,\infty)$. Additionally, if $x,y\in\supp\mu$, $x\neq y$, then $(x,y)\in A_{R_k}$ for some $k\in\N$ and the above argument shows that \eqref{eq:eisyklir} holds with $\delta_k$. Since $h$ is nondecreasing and $\rho(\sfd(x,y))\leq \delta_k$, we see that \eqref{eq:eisyklir} holds with $\rho(\sfd(x,y))$ as well.
\end{proof}

In the next lemma we see that the condition (2) of Theorem \ref{thm:planestimate} quite easily implies the condition (1) for almost every
pair of points $(x,y)$. In the proof of Theorem \ref{thm:planestimate} we will then improve this to hold for every pair of points.

\begin{lemma}\label{lma:almost1}
 Suppose that $(X,\sfd)$, $h$ and $\mu$ are such that (2) of Theorem \ref{thm:planestimate} holds.
 Then 
 for every $x \in \supp\mu$ we have that
 for $\mu$-almost every $y \in \supp\mu$ there exists $(z_i)_{i=1}^N\subset\supp\mu$ such that \eqref{eq:eisykli} holds.
\end{lemma}
\begin{proof}
 Suppose that the claim is false.
 Then there exists $y \in \supp\mu$ such that 
\[
 A:= \{x \in \supp\mu \,:\, \text{\eqref{eq:eisykli} fails for all }(z_i)_{i=1}^N \subset \supp\mu\}
\]
has positive $\mu$-measure.
For all $0 < t < \mu(A)$ consider the measure
\[
 \nu_t := \mu\restr{X \setminus A} + \frac{\mu(A)-t}{\mu(A)}\mu\restr{A} + t \delta_y
\]
and the transport $\lambda_t \in \Pi(\mu,\nu_t)$ defined as
\[
 \lambda_t := (\id,\id)_\sharp\big((1-\frac{t}{\mu(A)}\chi_A)\mu\big) + \frac{t}{\mu(A)}(\mu\restr{A}\times\delta_y).
\]
Then
\begin{equation}\label{eq:lambdasupport}
 \supp\lambda_t = \{(x,x)\,:\,y \in \supp\mu\} \cup (A \times \{y\}).
\end{equation}
In order to see that $\lambda_t$ is optimal, we check that $\supp\lambda_t$ is $h\circ\sfd$-cyclically monotone.
To observe that this is the case, take $(x_i,y_i)_{i=1}^M \subset \supp \lambda_t$ and let $\sigma$ be a permutation of $\{1,\dots,M\}$.
By \eqref{eq:lambdasupport}, either $x_i = y_i$ or $y_i = y$ for all $i$. Write $I = \{i \,:\, x_i \ne y_i \}$.
For each $i \in I$ define a finite sequence $(i_j)_{j=0}^{N_i}$ by setting $i_0 = i$, $i_1 = \sigma(i)$,
and inductively $i_j = \sigma(i_{j-1})$ if $\sigma(i_{j-1})\notin I$.
Now by the definition of $A$ we get
\[
 \sum_{i=1}^M h\circ\sfd(x_i,y_i) = \sum_{i\in I} h\circ\sfd(x_i,y_i)
 \le \sum_{i\in I} \sum_{j=0}^{N_i}h\circ\sfd(x_{i_j},y_{\sigma(i_j)})
 \le \sum_{i=1}^Mh\circ\sfd(x_{i},y_{\sigma(i)}).
\]
Thus $\supp\lambda_t$ is $h\circ\sfd$-cyclically monotone.

To see that (2) fails notice that 
\[
 \lambda_t-\esup \sfd = \mu-\esup_{y\in A}\sfd(x,y) > 0,
\]
but
\[
\int h\circ \sfd\,\d\lambda_t = \frac{t}{\mu(A)}\int_A h(\sfd(x,y))\,\d\mu(y) \to 0, \qquad \text{as }t\to 0.
\] 
\end{proof}

Now we have all the tools to prove Theorem \ref{thm:planestimate}.

\begin{proof}[Proof of Theorem \ref{thm:planestimate}]
Let us first show that (1) implies (2).
Thus we assume that $\mu\in\prob{X}$ has compact support and satisfies condition \eqref{eq:eisykli}. To prove the claim we fix $\nu\in\prob{\supp\mu}$ and an $h\circ\sfd$-optimal plan $\lambda\in\Pi(\mu,\nu)$. For simplicity denote $D:=\lambda-\esup \sfd$. We may assume $D>0$.

The key part of the proof is to show that
\begin{equation}\label{eq:larvio}
\lambda\big( \{(x,y)\in \supp\mu\times\supp\mu: \sfd(x,y)\geq \frac{\rho(D)}{4} \} \big) \geq m(\frac{\rho(D)}{4}).
\end{equation}
This implies the inequality
\[
\int h\circ\sfd \,\d\lambda \geq h(\frac{\rho(D)}{4})\,m(\frac{\rho(D)}{4}),
\]
which is exactly what we want to prove.

So, let us prove \eqref{eq:larvio}. We do this by contradiction, i.e. we assume that
\begin{equation}\label{eq:eilarvio}
\lambda\big( \{(x,y)\in \supp\mu\times\supp\mu: \sfd(x,y)\geq \frac{\rho(D)}{4} \} \big) < m(\frac{\rho(D)}{4}).
\end{equation}
Since $\supp\lambda$ is compact, we find $(x,y)\in\supp\lambda$ such that $\sfd(x,y)=D$. Then we apply Lemma \ref{lma:rhoexists} and find a sequence $(z_i)_{i=0}^{N+1}\subset\supp\mu$ such that $z_0=x$, $z_{N+1}=y$ and
\begin{equation}\label{eq:rho}
\sum_{i=0}^N h\big(\sfd(z_i,z_{i+1})+\rho(D)\big) \;<\; h(\sfd(x,y)).
\end{equation}

Next we can use an argument similar to the one used in Lemma \ref{lma:bigpartlongway} and show that \eqref{eq:eilarvio} implies that for every $i \in \{1,\ldots, N\}$ we have
\begin{align*}
\lambda\big( B(z_i,\frac{\rho(D)}{2})\times B(z_i,\frac{\rho(D)}{2}) \big) & \;\geq \;  \mu(B(z_i,\frac{\rho(D)}{4}))-\lambda\big( B(z_i,\frac{\rho(D)}{4})\times (X\setminus B(z_i,\frac{\rho(D)}{2}))\big) \\
& \;>\; m(\frac{\rho(D)}{4})-m(\frac{\rho(D)}{4}) \;=\; 0.
\end{align*}
Thus for every $i\in\{1,\ldots, N\}$ there exists $(x_i,y_i)\in\supp\lambda \cap \big(B(z_i,\frac{\rho(D)}{2})\times B(z_i,\frac{\rho(D)}{2})\big)$.

Finally, since $x_i,y_i\in B(z_i,\frac{\rho(D)}{2})$ and \eqref{eq:rho} holds, we have
\begin{align*}
 h(\sfd(x,y_1))&+\sum_{i=1}^{N-1}h(\sfd(x_i,y_{i+1}))+h(\sfd(x_N,y)) \\
\leq \; & h\big(\sfd(x,z_1)+\rho(D)\big)+\sum_{i=1}^{N-1}h\big(\sfd(z_i,z_{i+1})+\rho(D)\big)+h\big(\sfd(z_{N},y)+\rho(D)\big) \\
< \; & h(\sfd(x,y)) \;\leq\; h(\sfd(x,y))+\sum_{i=1}^Nh(\sfd(x_i,y_i)).
\end{align*}
This is a contradiction, since $(x_1,y_1),\ldots,(x_N,y_N),(x,y)\in\supp\lambda$ and Theorem \ref{thm:optimaliscyclical} implies that $\lambda$ is $h\circ\sfd$-cyclically monotone.
\medskip

Let us then show that (2) implies (1).
First note that compactness of $\supp\mu$ follow from (2) by Theorem \ref{thm:main}.
Suppose that (1) does not hold. Then there exists a pair of points $x,y\in\supp\mu$, $x\neq y$, such that 
\[
\inf_{z\in\supp\mu}f_{x,y}(z) = h(\sfd(x,y)), 
\]
where 
\[
f_{x',y'}(z_1):= \inf\bigg\{ \sum_{j=0}^N h(\sfd(z_j,z_{j+1})): z_0=x',z_{N+1}=y', (z_j)_{j=2}^N\subset\supp\mu, N\in\mathbb{N} \bigg\}.
\]
Define $E:=\big\{ z\in\supp\mu: f_{x,y}(z) = h(\sfd(x,y))\big\}$, and for every $k \in \N$,
\[
F_k:=\Big\{ z\in\supp\mu: f_{x,y}(z) \ge h(\sfd(x,y))+\frac{1}{k} \Big\} \quad\textrm{and} \quad E_k:=\supp\mu\setminus F_k.
\]
Due to the uniform continuity of $h$ in $[0,\diam(\supp\mu)]$ the function $f_{x,y}$ is continuous as a function of $z$.
Thus the sets $F_k$ and $E$ are compact.
Also we see that $x\in E\subset E_k$, and so we get $\dist(x,F_k)>0$ for all $k \in \N$ by compactness. Consequently we also have $\mu(E_k) > 0$ for all $k \in \N$.

We claim that $\mu(E)=0$. This can be proven by contradiction: For every $z_1\in E$ we have by definition $f_{x,y}(z_1)=h(\sfd(x,y))$, which implies
\[
h(\sfd(x,z_1))=\inf\bigg\{ \sum_{j=0}^N h(\sfd(w_j,w_{j+1})): w_0=x, w_{N+1}=z_1, (w_j)_{j=1}^N\subset\supp\mu, N\in\mathbb{N} \bigg\}.
\]
If $\mu(E)>0$, this yields a contradiction with Lemma \ref{lma:almost1}.

Take $\nu_k=\mu\restr{F_k}+\mu(E_k)\delta_y$ and let $\lambda_k\in\Pi(\mu,\nu_k)$ be an $h\circ \sfd$-optimal plan. Then
\[
\int h\circ \sfd \,\d\lambda_n \le h(\diam(\supp\mu))\mu(E_k) \to 0, \quad\textrm{as }\, k\to\infty,
\]
since $\lim_k \mu(E_k)=\mu(E)=0$. Now we prove that $\lambda_k-\esup \sfd \ge {\sfd(x,y)}/{2}$, causing a contradiction with condition (2).

We begin by defining the following sequence of sets: First we define $A_0:=\{y\}$ and then recursively
\[
A_{j+1}:=\big\{z\in \supp\mu: (z,w)\in\supp\lambda_k \textrm{ for some } w\in A_j\big\} \quad\textrm{and}\quad A:=\bigcup_{j=0}^\infty A_j.
\]
These definitions ensure that $\lambda_k(A_{j+1}\times A_j)=\lambda_k(X\times A_j)$, and hence we get
\begin{align*}
\mu(E_k\cap A)+ & \mu(F_k\cap A) = \mu(A) = \lambda_t\big(\bigcup_{j=0}^\infty (A_j\times X)\big) \ge \lambda_t\big(\bigcup_{j=0}^\infty (A_{j+1}\times A_j)\big) \\
 = & \lambda_t\big(\bigcup_{j=0}^\infty (X\times A_j)\big) = \nu_t(A) = \mu(F_k\cap A)+\mu(E_k).
\end{align*}
This implies that $\mu(E_k\setminus A)=0$.

Now we note that the function
\[
g(x')=\inf_{z\in F_k}f_{x',y}(z)-h(\sfd(x',y)),
\]
is continuous and $g(x) \ge 1/k$. Thus there exists $0<r<\sfd(x,y)/2$ such that 
\begin{equation}\label{eq:zeta1}
\inf_{z\in F_k}f_{x',y}(z)-h(\sfd(x',y))=g(x')>0 \quad \textrm{whenever}\; x'\in B(x,r).
\end{equation}
Furthermore we may choose $r>0$ above to be smaller than $\dist(x,F_k)$, so that $B(x,r)\cap\supp\mu\subset E_k$.

Finally, since $\mu(B(x,r))>0$, $B(x,r)\cap\supp\mu\subset E_k$ and $\mu(E_k\setminus A)=0$, there exists $w_{N+1}\in B(x,r)\cap A_{N+1}$ for some $N$. Furthermore, using the definitions of the sets $A_j$ we find a sequence $(w_j)_{j=0}^{N+1}$ in $X$
such that $(w_{j+1},w_j)\in \supp\lambda_k \cap A_{j+1}\times A_j$ for all $j \in \{0,1, \ldots,N\}$. If we reverse the order of these points by defining $z_j=w_{N+1-j}$, we have $z_0\in B(x,r)$, $z_{N+1}=y$ and $(z_j)_{j=1}^N\subset\supp\nu_k$. Now, if $z_1\in F_k$, then it follows from \eqref{eq:zeta1} that $f_{z_0,y}(z_1)-h(\sfd(z_0,y)) \ge g(z_0)>0$. This means that in particular
\[
\sum_{j=0}^N h(\sfd(z_j,z_{j+1})) > h(\sfd(z_0,z_{N+1})).
\]
This is a contradiction, since $\lambda_k$ is $h\circ\sfd$-cyclically monotone and $(z_j,z_{j+1})\in\supp\lambda_k$. 
Thus we have $z_1\in\supp\nu_k\setminus F_k=\{y\}$ and $\lambda_k-\esup \sfd \ge \sfd(z_0,z_1) \ge {\sfd(x,y)}/2$.
\end{proof}

Let us now provide a criterion for the cost function $h$ and the measure $\mu$ which guarantee
that the condition (1) of Theorem \ref{thm:planestimate} holds.
Here $\mathcal{H}^s$ denotes the $s$-dimensional Hausdorff measure.

\begin{proposition}\label{prop:nicelyconnectedsupport}
 Let $1 \le  s < \infty$. Suppose that $h \colon [0,\infty) \to [0,\infty)$ satisfies
 \[
  \lim_{t \downarrow 0}\frac{h(t)}{t^s} = 0
 \]
 and $\mu \in \prob{X}$ is such that for every $x,y \in \supp\mu$ there exists a curve $\gamma \subset \supp\mu$
 connecting $x$ to $y$ with $\mathcal{H}^s(\gamma) < \infty$. Then condition (1) of Theorem \ref{thm:planestimate} holds.
 
 In particular, condition (1) of Theorem \ref{thm:planestimate} holds if $h(0) = h'(0)=0$ and if any $x,y \in \supp\mu$
 can be connected by a rectifiable curve in $\supp\mu$.
\end{proposition}
\begin{proof}
 Let $x,y \in \supp\mu$, $x \ne y$, and let $\gamma \subset \supp\mu$ be a curve connecting $x$ to $y$ with $\mathcal{H}^s(\gamma) < \infty$.
 Define 
 \[
  \varepsilon := 2^{-2-s}\frac{h(\sfd(x,y))}{\mathcal{H}^s(\gamma)}.
 \]
 Let $\delta > 0$ be small enough so that $h(t) < \varepsilon t^s$ for all $0 < t < \delta$.
 Since $\mathcal{H}^s(\gamma) < \infty$ and $\gamma$ is compact, there exists a finite 
 collection of relatively open sets $E_i \subset \supp\mu$, $i \in \{1,\dots,N\}$, with 
 \[
\diam(E_i) < \frac{\delta}2 \text{ for all }i\quad,\quad \gamma \subset \bigcup_{i=1}^N E_i\quad\text{ and }\quad\sum_{i=1}^N \diam(E_i)^s < 2 \mathcal{H}^s(\gamma).  
 \]
 Since $\gamma$ is connected there exists a sequence $\{i_j\}_{j=1}^K$ such that $x \in E_{i_1}$, $y \in E_{i_K}$,
 $i_j \ne i_{j'}$ if $j \ne j'$ and $E_{i_j} \cap E_{i_{j+1}} \ne \emptyset$ for all $j \in \{1,\dots,K-1\}$.
 Define $z_1 = x$ and $z_K = y$. For each $j \in \{2,\dots,K-1\}$ select a point $z_j \in E_{i_j}$.
 Now for all $j \in \{1,\dots,K-1\}$ we have
 \[
  \sfd(z_j,z_{j+1}) \le \diam(E_{i_j}) + \diam(E_{i_{j+1}}) < \frac{\delta}2 + \frac{\delta}2 = \delta.
 \]
 Therefore
 \begin{align*}
   \sum_{j=1}^Kh(\sfd(z_j,z_{j+1})) & \le \sum_{j=1}^K \varepsilon \sfd^s(z_j,z_{j+1})
  \le \sum_{j=1}^K \varepsilon \left((2\diam(E_{i_j}))^s + (2\diam(E_{i_{j+1}}))^s\right)\\
  & \le 2^{1+s}\varepsilon\sum_{i=1}^N \diam(E_i)^s < 2^{2+s}\varepsilon\mathcal{H}^s(\gamma) = h(\sfd(x,y))
 \end{align*}
 and so condition (1) of Theorem \ref{thm:planestimate} holds.
\end{proof}

%

\end{document}